\documentclass[12pt]{article}
\usepackage{amssymb,amsthm,amsmath,latexsym}
\usepackage{url}
\usepackage{lscape}
\newtheorem{thm}{Theorem}[section]
\newtheorem{prop}[thm]{Proposition}
\newtheorem{lem}[thm]{Lemma}

\theoremstyle{remark}

\newcommand{\ZZ}{\mathbb{Z}}

\newcommand{\allzero}{\mathbf{0}}

\DeclareMathOperator{\wt}{wt}

\begin{document}
\title{On the classification of $\ZZ_4$-codes}

\author{
Makoto Araya\thanks{Department of Computer Science,
Shizuoka University,
Hamamatsu 432--8011, Japan.
email: araya@inf.shizuoka.ac.jp},
Masaaki Harada\thanks{
Research Center for Pure and Applied Mathematics,
Graduate School of Information Sciences,
Tohoku University, Sendai 980--8579, Japan.
email: mharada@m.tohoku.ac.jp.},
Hiroki Ito\thanks{
Koki Consultant Inc., Kitakata 966--0902, Japan.
This work was carried out at 
Graduate School of Information Sciences,
Tohoku University.}
and
Ken Saito\thanks{
Research Center for Pure and Applied Mathematics,
Graduate School of Information Sciences,
Tohoku University, Sendai 980--8579, Japan.
email: kensaito@ims.is.tohoku.ac.jp.}
}

\maketitle

\begin{abstract}
In this note, we study the classification of $\ZZ_4$-codes.
For some special cases $(k_1,k_2)$, 
by hand, we give a classification of $\ZZ_4$-codes of length 
$n$ and type $4^{k_1}2^{k_2}$ satisfying a certain condition.
Our exhaustive computer search completes
the classification of $\ZZ_4$-codes of lengths up to $7$.
\end{abstract}

\section{Introduction}

Let $\ZZ_4\ (=\{0,1,2,3\})$ denote the ring of integers
modulo $4$.
A {\em $\ZZ_{4}$-code} $C$ of length $n$ 
is a $\ZZ_{4}$-submodule of $\ZZ_{4}^n$.
Over the past decade codes over finite rings have gained
in importance for both practical and theoretical reasons. 
In particular, there has been interest in $\ZZ_4$-codes.
For example, a simple relationship between 
the best known nonlinear binary codes such as the 
Kerdock, Preparata, Goethals codes, 
which contain more codewords
than any known linear codes with the same minimum distance and 
$\ZZ_4$-codes was discovered by Hammons, Kumar, Calderbank, 
Sloane and Sol\'e~\cite{Z4-HKCSS}.

The 
{\em Hamming weight} $\wt_H(x)$,
{\em Lee weight} $\wt_L(x)$ and
{\em Euclidean weight} $\wt_E(x)$ 
of a codeword $x$ of a $\ZZ_4$-code $C$ are defined as
$n_1(x)+n_2(x)+n_3(x)$,
$n_1(x)+2n_2(x)+n_3(x)$ and
$n_1(x)+4n_2(x)+n_3(x)$, respectively, 
where $n_i(x)$ is the number of components of $x$ which
are equal to $i$.
The {\em minimum Hamming weight} $d_H(C)$,
{\em minimum Lee weight} $d_L(C)$ and 
{\em minimum Euclidean weight} $d_E(C)$
of $C$
is the smallest Hamming, Lee and Euclidean weight among
all non-zero codewords of $C$, respectively.
A $\ZZ_4$-code $C$ of length $n$ and type $4^{k_1}2^{k_2}$ is called 
{Hamming-optimal}, {Lee-optimal} and {Euclidean-optimal} 
if $C$ has 
the largest minimum Hamming weight,
the largest minimum Lee weight and 
the largest minimum Euclidean weight
among all $\ZZ_4$-codes of length $n$ and type $4^{k_1}2^{k_2}$, respectively.

Two $\ZZ_4$-codes $C$ and $C'$ are {\em equivalent}, 
denoted $C \cong C'$, if $C$ can be obtained from $C'$
by permuting the coordinates and (if necessary) changing
the signs of certain coordinates.
It is a fundamental problem to classify $\ZZ_4$-codes for modest
lengths, up to equivalence.
Concerning self-dual $\ZZ_4$-codes, 
the classification is known for lengths up to $19$
(see~\cite{HM}).
Beyond self-dual $\ZZ_4$-codes, 
only a few results on the classification of $\ZZ_4$-codes
are known~\cite{DGPW} and \cite{GW} (see also~\cite{Wong}).
More precisely,
a classification of optimal $\ZZ_4$-codes $C$ of length $n \le 8$ with $|C|=2^n$
was done in~\cite{GW}, noting that the definition of optimal $\ZZ_4$-codes
is slightly different from the above definition.
A classification of $\ZZ_4$-codes which are
Hamming-optimal, Lee-optimal and Euclidean-optimal 
was presented in~\cite{DGPW} for lengths up to $7$.
In this note, we study the classification of $\ZZ_4$-codes.

This note is organized as follows.
In Section~\ref{sec:Pre}, some preliminaries are given.
In Section~\ref{sec:trivial},  the notion of trivial extensions is introduced.
We show that it is sufficient to consider only 
inequivalent $\ZZ_4$-codes of length $n$ and type $4^{k_1}2^{k_2}$, 
none of which is equivalent
to the trivial extension of a $\ZZ_4$-code of length $n-1$,
in order to complete the classification of 
$\ZZ_4$-codes of length $n$ and type $4^{k_1}2^{k_2}$
(Proposition~\ref{prop:trivial}).
In Section~\ref{sec:sp}, for some special cases $(k_1,k_2)$, 
by hand, we give a classification of $\ZZ_4$-codes of length 
$n$ and type $4^{k_1}2^{k_2}$, none of which is equivalent
to the trivial extension of a $\ZZ_4$-code of length $n-1$.
In Section~\ref{sec:method}, 
we describe a computer-aided classification method of $\ZZ_4$-codes.
In Section~\ref{sec:result}, our exhaustive computer search completes
the classification of $\ZZ_4$-codes 
of lengths $n \le 7$, none of which is equivalent
to the trivial extension of a $\ZZ_4$-code of length $n-1$,
by the method given in Section~\ref{sec:method}.
Along with Proposition~\ref{prop:trivial},
we complete the classification of 
all $\ZZ_4$-codes of lengths up to $7$.
All computer calculations in this note
were done by {\sc Magma}~\cite{Magma}.

\section{Preliminaries}\label{sec:Pre}

Let $C$ be a $\ZZ_4$-code of length $n$.
It is known that 
$C$ is permutation-equivalent to a $\ZZ_4$-code with 
generator matrix of the form:
\begin{equation}
\label{eq:g-matrix}
\left(\begin{array}{ccc}
I_{k_1} & A & B \\
O    &2I_{k_2} & 2D \\
\end{array}\right),
\end{equation}
where $I_k$ denotes the identity matrix of order $k$,
$O$ denotes the zero matrix,
$A$ and $D$ are $(1,0)$-matrices, 
and $B$ is a $\ZZ_4$-matrix of appropriate sizes.
We say that $C$ has {\em type~$4^{k_1}2^{k_2}$} 
(see~\cite{Z4-C-S} and \cite{Z4-HKCSS}).

The {\em dual code} ${C}^\perp$ of $C$ is defined as
${C}^\perp = \{ x \in \ZZ_{4}^n \mid x \cdot y = 0$ for 
all $y \in C\}$,
where $x \cdot y$ is the standard inner product.
The dual code $C^\perp$ of the $\ZZ_4$-code $C$ with generator 
matrix~\eqref{eq:g-matrix}
has the following generator matrix:
\[
\left(\begin{array}{ccc}
-B^T-D^TA^T & D^T & I_{n-k_1-k_2} \\
2A^T & 2I_{k_2} & O
\end{array}\right),
\]
where $A^T$ denotes the transposed matrix of $A$~\cite{Z4-C-S}.
This means that $C^\perp$ has type $4^{n-k_1-k_2}2^{k_2}$.
Throughout this note, let $N(n,k_1,k_2)$ denote the number of 
inequivalent $\ZZ_4$-codes of length $n$
and type $4^{k_1}2^{k_2}$.
Then it is trivial that
\begin{equation}\label{eq:number}
N(n,k_1,k_2)=N(n,n-k_1-k_2,k_2). 
\end{equation}

For an element $a \in \ZZ_4$, we denote the binary element $a \pmod 2$
by $\hat{a}$.
The residue code $C^{(1)}$ of a $\ZZ_4$-code $C$ of length $n$
is defined as the following binary code:
\[
C^{(1)} = \{(\hat{c_1},\hat{c_2},\ldots,\hat{c_n})
\mid (c_1,c_2,\ldots,c_n) \in C \}.
\]
If $C$ is a $\ZZ_4$-code of length $n$ and type~$4^{k_1}2^{k_2}$
having generator matrix~\eqref{eq:g-matrix}, then
$C^{(1)}$ is a binary $[n,k_1]$ code with the following
generator matrix:
\begin{equation}
\label{eq:residue}
\left(\begin{array}{ccc}
I_{k_1} & A &  (\hat{b_{ij}})\\
\end{array}\right),
\end{equation}
where $B=(b_{ij})$ and $A$ is regarded as a binary matrix.
Note that equivalent $\ZZ_4$-codes have equivalent residue codes.

The Hamming, Lee and symmetrized weight enumerators 
of a $\ZZ_4$-code $C$ of length $n$ are defined as:
\begin{align*}
hwe_C(x,y)=&\sum_{c \in C} x^{n-\wt_H(c)}y^{\wt_H(c)},\\
lwe_C(x,y)=&\sum_{c \in C} x^{2n-\wt_L(c)}y^{\wt_L(c)},\\
swe_C(x,y,z)=&\sum_{c \in C} x^{n_0(c)}y^{n_1(c)+n_3(c)} z^{n_2(c)},
\end{align*}
respectively.
Note that equivalent $\ZZ_4$-codes have 
identical Hamming, Lee, symmetrized weight enumerators.

\section{Trivial extensions}\label{sec:trivial}

In this section, we introduce the notion of trivial extensions.  
Some basic facts are given.

Let $C$ be a $\ZZ_4$-code of length $n$.
Define the $\ZZ_4$-code of length $n+1$:
\[
\overline{C}=\{(c,0) \mid c \in C\}.
\]
We say that $\overline{C}$ is a {\em trivial extension} of $C$.

\begin{lem}\label{lem:equiv}
Let $C$ and $C'$ be $\ZZ_4$-codes of length $n$.
Suppose that $D$ and $D'$ are $\ZZ_4$-codes of length $n+1$
satisfying that 
$\overline{C} \cong D$ and $\overline{C'} \cong D'$.
Then $C \cong C'$ if and only if $D \cong D'$.
\end{lem}
\begin{proof}
Suppose that $D \cong D'$. 
Then $\overline{C} \cong \overline{C'}$. 
Since the last coordinate of each codeword of $\overline{C}$ and $\overline{C'}$ is $0$, 
there is a $(1,-1,0)$-monomial matrix $P$ of order $n+1$ such that
$\overline{C}=\overline{C'}P$, where
\[
P=
\left(\begin{array}{ccccccccc}
 && & 0 \\
 & P'&& \vdots \\
 && &  0\\
0&\cdots& 0&  1
\end{array}\right).
\]
This gives that $C=C'P'$, thus $C \cong C'$.
The converse is immediate.
\end{proof}

\begin{prop}\label{prop:trivial}
Let ${\cal C}(n,k_1,k_2)$ be a set of all inequivalent $\ZZ_4$-codes
of length $n$ and type $4^{k_1}2^{k_2}$.
Then
there are
a set ${\cal C}(n-1,k_1,k_2)$ of all inequivalent $\ZZ_4$-codes
of length $n-1$ and type $4^{k_1}2^{k_2}$,
and
a set ${\cal D}(n,k_1,k_2)$  of all inequivalent $\ZZ_4$-codes
of length $n$ and type $4^{k_1}2^{k_2}$, none of which is equivalent
to the trivial extension of a $\ZZ_4$-code of length $n-1$,
such that
\begin{align*}
& {\cal C}(n,k_1,k_2)={\cal D}(n,k_1,k_2) \cup 
  \{\overline{C} \mid C \in {\cal C}(n-1,k_1,k_2)\},
\\&
{\cal D}(n,k_1,k_2) \cap 
\{\overline{C} \mid C \in {\cal C}(n-1,k_1,k_2)\}
=\emptyset.
\end{align*}
\end{prop}
\begin{proof}
It follows from the definition that 
${\cal C}(n,k_1,k_2)$ is the direct sum of ${\cal D}(n,k_1,k_2)$ and the set 
${\cal E}(n,k_1,k_2)$
of inequivalent $\ZZ_4$-codes
of length $n$ and type $4^{k_1}2^{k_2}$, which are equivalent
to the trivial extensions of some $\ZZ_4$-codes of length $n-1$.
By Lemma~\ref{lem:equiv}, there is a one-to-one correspondence between
${\cal E}(n,k_1,k_2)$ and ${\cal C}(n-1,k_1,k_2)$. 
The result follows.
\end{proof}

Hence, it is sufficient to consider only 
inequivalent $\ZZ_4$-codes of length $n$ and type $4^{k_1}2^{k_2}$, 
none of which is equivalent
to the trivial extension of a $\ZZ_4$-code of length $n-1$,
in order to complete the classification of 
$\ZZ_4$-codes of length $n$ and type $4^{k_1}2^{k_2}$.

\section{Classification for special cases}\label{sec:sp}

In this section, for some special cases $(k_1,k_2)$, 
by hand, we give a classification of $\ZZ_4$-codes of length 
$n$ and type $4^{k_1}2^{k_2}$, none of which is equivalent
to the trivial extension of a $\ZZ_4$-code of length $n-1$.

Throughout this note, let $N'(n,k_1,k_2)$ denote the number of 
inequivalent $\ZZ_4$-codes
of length $n$ and type $4^{k_1}2^{k_2}$, none of which is equivalent
to the trivial extension of a $\ZZ_4$-code of length $n-1$.

\begin{prop}
$N'(n,n,0)=N'(n,0,n)=N'(n,0,1)=1$.
\end{prop}
\begin{proof}
Let $C_{k_1,k_2}$ be a $\ZZ_4$-code of length $n$ and 
type $4^{k_1}2^{k_2}$, 
which is inequivalent to the trivial extension of a $\ZZ_4$-code of length $n-1$.
Then
$C_{n,0},C_{0,n}$ and $C_{0,1}$ are equivalent to the $\ZZ_4$-codes with generator matrices 
$I_n$, $2I_n$ and  $(2\ \cdots\ 2)$, respectively.
\end{proof}

\begin{prop}
$N'(n,n-1,1)=n$.
\end{prop}
\begin{proof}
Let $C$ be a $\ZZ_4$-code of length $n$ and type $4^{n-1}2^1$,
which is inequivalent to the trivial extension of a $\ZZ_4$-code of length $n-1$.
Then $C$ has generator matrix of the form:
\[
G=
\left(\begin{array}{cccccc}
       &            &      & a_1\\
       & I_{n-1}  &       & \vdots\\
       &            &      & a_{n-1}\\
0     & \cdots & 0  &  2
\end{array}\right),
\]
where $a_i \in \{0,1\}$ $(i=1,2,\ldots,n-1)$.
The generator matrix of the dual code $C^\perp$ of the $\ZZ_4$-code 
$C$ with generator matrix $G$ is given by
$G^\perp=(2a_1\ \cdots\ 2a_{n-1}\ 2)$.
We may suppose without loss of generality that $a_i \le a_{i+1}$
$(i=1,2,\ldots,n-2)$,
where $0 < 1$.
Hence, $N'(n,n-1,1) \le n$.
Let $C^\perp_m$ be the $\ZZ_4$-code with generator matrix 
$G^\perp$, where
$m$ is the number of $i \in \{1,2,\ldots,n-1\}$ with $a_i=1$ in $G^\perp$.
If $m_1 \ne m_2$, then 
$C^\perp_{m_1}$ and $C^\perp_{m_2}$ are inequivalent.
This yields that $N'(n,n-1,1)= n$.
\end{proof}

\begin{prop}\label{prop:residue}
$N'(n,1,n-1)=n$.
\end{prop}
\begin{proof}
Let $C$ be a $\ZZ_4$-code of length $n$ and type $4^12^{n-1}$,
which is inequivalent to the trivial extension of a $\ZZ_4$-code of length $n-1$.
Then $C$ has generator matrix of the form:
\[
G=
\left(\begin{array}{ccccc}
1         & a_1 & \cdots & a_{n-1} \\
0         &       &            & \\
\vdots &       & 2I_{n-1} & \\
0         &       &            & \\
\end{array}\right),
\]
where $a_i \in \{0,1\}$ $(i=1,2,\ldots,n-1)$.
We may suppose without loss of generality that $a_i \le a_{i+1}$
$(i=1,2,\ldots,n-2)$,
where $0 < 1$.
Hence, $N'(n,1,n-1) \le n$.

Now we consider the residue code $C^{(1)}$ of the $\ZZ_4$-code $C$.
The residue code $C^{(1)}$ has generator matrix
$G^{(1)}=(1\ a_1\ \cdots \ a_{n-1})$ 
(see~\eqref{eq:residue} for the generator matrix of $C^{(1)}$).
Let $C^{(1)}_m$ be the residue code with generator matrix 
$G^{(1)}$, where
$m$ is the number of $i \in \{1,2,\ldots,n-1\}$ with $a_i=1$ in $G^{(1)}$. 
If $m_1 \ne m_2$, then 
$C^{(1)}_{m_1}$ and $C^{(1)}_{m_2}$ are inequivalent.
This yields that $N'(n,1,n-1) = n$.
\end{proof}

\begin{prop}
$N'(n,1,0)=n$.
\end{prop}
\begin{proof}
Let $C$ be a $\ZZ_4$-code of length $n$ and type $4^12^0$,
which is inequivalent to the trivial extension of a $\ZZ_4$-code of length $n-1$.
Then $C$ has generator matrix of the form
$(1\ a_1\  \cdots\ a_{n-1})$, where $a_i \in \{1,2,3\}$
$(i=1,2,\ldots,n-1)$.
We may suppose without loss of generality that $a_i \in \{1,2\}$ $(i=1,2,\ldots,n-1)$
and $a_i \le a_{i+1}$ $(i=1,2,\ldots,n-2)$, where $1 < 2$.
Hence, $N'(n,1,0) \le n$.

Now consider the residue code $C^{(1)}$ of $C$.
By an argument similar to the last part of the proof of the above proposition, 
$N'(n,1,0) = n$.
\end{proof}

\begin{prop}
$N'(n,0,n-1)=n-1$.
\end{prop}
\begin{proof}
Let $C$ be a $\ZZ_4$-code of length $n$ and type $4^02^{n-1}$,
which is inequivalent to the trivial extension of a $\ZZ_4$-code of length $n-1$.
Then $C$ has generator matrix of the form:
\[
G=
\left(\begin{array}{ccccc}
 &  & 2a_1 \\
     & 2I_{n-1} & \vdots \\
       &            & 2a_{n-1}\\
\end{array}\right),
\]
where $a_i \in \{0,1\}$ $(i=1,2,\ldots,n-1)$ and $a_j=1$ 
for some $j \in \{1,2,\ldots,n-1\}$.
The generator matrix of the dual code $C^\perp$ of the $\ZZ_4$-code 
$C$ with generator matrix $G$ is given by:
\[
\left(\begin{array}{cccccc}
a_1 & \cdots & a_{n-1} & 1\\
      &            &      & 0\\
       & 2I_{n-1} &  & \vdots\\
       &            &     & 0\\
\end{array}\right).
\]
We may suppose without loss of generality that $a_i \le a_{i+1}$ 
$(i=1,2,\ldots,n-2)$,
where $0 < 1$.
Hence, $N'(n,0,n-1) \le n-1$.

Now consider the residue code ${C^\perp}^{(1)}$ of the dual code $C^\perp$.
An argument similar to the last part of 
the proof of Proposition~\ref{prop:residue} shows 
the existence of $n-1$ inequivalent dual codes $C^\perp$.
Hence, $N'(n,0,n-1)=n-1$.
\end{proof}

\begin{prop}
$N'(n,n-1,0) = n(n+1)/2 -1$. 
\end{prop}
\begin{proof}
Let $C$ be a $\ZZ_4$-code of length $n$ and type $4^{n-1}2^{0}$,
which is inequivalent to the trivial extension of a $\ZZ_4$-code of length $n-1$.
Then $C$ has generator matrix of the form:
\[
G=
\left(\begin{array}{ccccc}
 &         & a_1 \\
 & I_{n-1} & \vdots \\
 &         & a_{n-1}\\
\end{array}\right),
\]
where $a_i \in \{0,1,2,3\}$ $(i=1,2,\ldots,n-1)$ and $a_j \ne 0$ 
for some $j \in \{1,2,\ldots,n-1\}$.
Here we may suppose without loss of generality that
$a_i \in \{0,2,3\}$ $(i=1,2,\ldots,n-1)$.
The generator matrix of the dual code $C^\perp$ of the $\ZZ_4$-code 
$C$ with generator matrix $G$ is given by
$(3a_1\  \cdots \ 3a_{n-1}\ 1)$.
We may suppose without loss of generality that $3a_i \le 3a_{i+1}$ 
$(i=1,2,\ldots,n-2)$,
where $0 < 1 < 2$.
The vector $x=(3a_1,3a_2,\ldots,3a_{n-1},1)$
has the form
$(0,\ldots,0,1,\ldots,1,2,\ldots,2,1)$ and
the vector is uniquely determined by the numbers $n_0(x)$ and $n_1(x)$.
Let $C_{m_0,m_1}$ be the $\ZZ_4$-code generated by the vector $x$
with $m_0=n_0(x)$ and $m_1=n_1(x)$.
Then $C^\perp$ is equivalent to one of the codes $C_{m_0,m_1}$
$(m_0=0,1,\ldots,n-2,m_1=0,1,\ldots,n-1-m_0)$.
Hence, we have
\begin{align*}
N'(n,n-1,0) 
\le &\sum_{m_0=0}^{n-2}(\sum_{m_1=0}^{n-1-m_0}1) 
= \frac{n(n+1)}{2} -1.  
\end{align*}

Now we consider the symmetrized weight enumerator.
The code $C_{m_0,m_1}$ has the following
symmetrized weight enumerator:
\[
swe_{C_{m_0,m_1}} (x,y,z)
= 
x^n 
+ 2 x^{m_0}y^{m_1}z^{n-m_0-m_1}
+ x^{n-m_1}z^{m_1}.
\]
If $(m_0,m_1) \ne (m'_0,m'_1)$, then 
$swe_{C_{m_0,m_1}} (x,y,z) \ne swe_{C_{m'_0,m'_1}} (x,y,z)$.
Hence, $C_{m_0,m_1}$ and $C_{m'_0,m'_1}$ are inequivalent.
This yields that $N'(n,n-1,0) = n(n+1)/2 -1$. 
\end{proof}

\section{Classification method}\label{sec:method}

In this section, we describe a computer-aided classification method
of $\ZZ_4$-codes.
We give some observations on generator matrices~\eqref{eq:g-matrix},
in order to reduce the number of generator matrices of $\ZZ_4$-codes,
which must be checked further for equivalences.

Define a lexicographical order on the vectors of $\ZZ_4^n$ as follows.
For $a=(a_1,a_2,\ldots,a_n)$ and $b=(b_1,b_2,\ldots,b_n) \in \ZZ_4^n$,
we define an order $a \le b$ if one of the following holds: 
\begin{itemize}
\item[(i)] $a_1 < b_1$,
\item[(ii)] there is an integer $k\in \{2,3,\ldots,n\}$ such that
$a_k < b_k$ and $a_i=b_i$ for all $i \in \{1,2,\ldots,k-1\}$,
\item[(iii)] $a=b$,
\end{itemize}
where $0 <1 <2<3$.
Let $M_{m \times n}(\mathbb{Z}_4)$ denote the set of all
$m \times n$ $\ZZ_4$-matrices.
For $T \subset M_{m \times n}(\mathbb{Z}_4)$, define the following:
\begin{align*}
&P_{row}(T)=\left\{ A=\begin{pmatrix}a_1  \\
\vdots \\
a_m
\end{pmatrix} \in T \mid  a_i \leq a_j\  (i \leq j) \right\},\\
&P_{col}(T)=\left\{ B=\begin{pmatrix}b_1\cdots b_n
\end{pmatrix} \in T \mid  b_i^T \leq b_j^T\  (i \leq j) \right\}.
\end{align*}
Let $C$ be a $\ZZ_4$-code of length $n$ and type $4^{k_1}2^{k_2}$,
having generator matrix of the form~\eqref{eq:g-matrix}.
We denote the matrix of the form~\eqref{eq:g-matrix} by
$G(A,B,D)$.
Then we consider the following sets of matrices:
\begin{align*}
&{\mathcal S}=\{ G(A,B,D)
\mid A \in M_{k_1 \times k_2}(\{0,1\}), 
B \in M_{k_1 \times \ell}(\mathbb{Z}_4), D \in M_{k_2 \times \ell}(\{0,1\}) \}, \\
&{\mathcal T}=\{ G(A,B,D) \in {\mathcal S}
\mid A \in P_{row}(M_{k_1 \times k_2}(\{0,1\}))\}, \\
&{\mathcal U}=\{ G(A,B,D) \in {\mathcal T}
\mid B \in {\mathcal B} \}, \\
&{\mathcal V}=\left\{ G(A,B,D) \in {\mathcal U} \mid \begin{pmatrix}
B \\
2D
\end{pmatrix}
\in P_{col}(M_{(k_1+k_2) \times \ell}(\mathbb{Z}_4)) \right\},
\end{align*}
where $\ell =n-k_1-k_2$ and
${\mathcal B}$ is the set of all $k_1 \times \ell$ $(0,2)$-matrices and 
all $k_1 \times \ell$ $\ZZ_4$-matrices $B$ satisfying the condition 
that the $i$-th row of $B$ contains entries only $0,1,2$ for the
smallest $i \in \{1,2,\ldots,k_1\}$ such that the $i$-th row of $B$ contains entries
except $0,2$.

\begin{lem}\label{lem:1}
If $G \in {\mathcal S}$, then there is a matrix $G' \in {\mathcal T}$
such that two $\ZZ_4$-codes with generator matrices $G$ and $G'$
are equivalent.
\end{lem}
\begin{proof}
By considering permutations of rows and columns,
one can obtain some generator matrix
$G(A',B',D) \in {\mathcal S}$ satisfying that
$A' \in P_{row}(M_{k_1 \times k_2}(\{0,1\}))$ and 
$B' \in M_{k_1 \times \ell}(\mathbb{Z}_4)$
from a generator matrix $G(A,B,D)$.
\end{proof}

\begin{lem}\label{lem:2}
If $G \in {\mathcal T}$, then there is a matrix $G' \in {\mathcal U}$
such that two $\ZZ_4$-codes with generator matrices $G$ and $G'$
are equivalent.
\end{lem}
\begin{proof}
By considering negations of some columns, 
one can obtain some generator matrix
$G(A,B',D) \in {\mathcal T}$ satisfying $B' \in {\mathcal B}$
from a generator matrix 
$G(A,B,D) \in {\mathcal T}$.
\end{proof}

\begin{lem}\label{lem:3}
If $G \in {\mathcal U}$, then there is a matrix $G' \in {\mathcal V}$
such that two $\ZZ_4$-codes with generator matrices $G$ and $G'$
are equivalent.
\end{lem}
\begin{proof}
By considering permutations of columns of  
$
\begin{pmatrix}
B \\
2D
\end{pmatrix}
$, 
one can obtain some generator matrix
$G(A,B',D')\in {\mathcal U}$ satisfying the following condition:
\[
\begin{pmatrix}
B' \\
2D'
\end{pmatrix}
\in P_{col}(M_{(k_1+k_2) \times \ell}(\mathbb{Z}_4))
\] 
from a generator matrix 
$G(A,B,D) \in {\mathcal U}$.
\end{proof} 

By Lemmas~\ref{lem:1}, \ref{lem:2} and \ref{lem:3},
we have the following:

\begin{prop}\label{prop:gm}
Let $C$ be a $\ZZ_4$-code of length $n$ and type $4^{k_1}2^{k_2}$,
which is inequivalent
to the trivial extension of a $\ZZ_4$-code of length $n-1$
and type $4^{k_1}2^{k_2}$.
Then there is a $\ZZ_4$-code $C'$ of length $n$ and type
$4^{k_1}2^{k_2}$ with $C \cong C'$, having 
generator matrix $G(A,B,D) \in \mathcal{V}$
satisfying that 
$\begin{pmatrix}
B \\
2D
\end{pmatrix}$ does not contain $\allzero^T$,
where $\allzero$ is the zero-vector.
\end{prop}

Proposition~\ref{prop:gm} substantially reduces the number of generator
matrices of $\ZZ_4$-codes 
which must be checked further for equivalences.

\section{$\ZZ_4$-codes of lengths up to 7}\label{sec:result}

A computer-aided classification method of $\ZZ_4$-codes
was given in the previous section.
By the method, 
in this section, we complete the classification of $\ZZ_4$-codes 
of lengths $n \le 7$, none of which is equivalent
to the trivial extension of a $\ZZ_4$-code of length $n-1$,
along with the results in Section~\ref{sec:sp}.

We consider $\ZZ_4$-codes of length $n$ and type $4^{k_1}2^{k_2}$, 
none of which is equivalent to the trivial extension of a 
$\ZZ_4$-code of length $n-1$, except in the case 
where the numbers $N'(n,k_1,k_2)$ were determined in Section~\ref{sec:sp}.
For a given set $(n,k_1,k_2)$,
by exhaustive search, we found all distinct $\ZZ_4$-codes of length
$n$ and type $4^{k_1}2^{k_2}$, 
satisfying the condition given in Proposition~\ref{prop:gm}.
Then the distinct codes can be divided into some classes 
by comparing the Hamming weight enumerators and
the Lee weight enumerators.
Of course, equivalent $\ZZ_4$-codes have 
identical Hamming weight enumerators and 
identical Lee weight enumerators.
To test equivalence of two $\ZZ_4$-codes $C$ and $C'$ in each class, 
we determined whether 
there is a $(1,-1,0)$-monomial matrix $P$ such that $C' = CP$ or not.
Then we completed the classification of $\ZZ_4$-codes
of length $n$ and type $4^{k_1}2^{k_2}$, none of which is equivalent
to the trivial extension of a $\ZZ_4$-code of length $n-1$ for
$n \le 7$.
The time required for the computer search of
the above classification of $\ZZ_4$-codes of lengths up to $7$,
which corresponds to one core of 
an Intel Xeon W5590 3.33GHz processor,
is approximately 1954 hours.
The parameter $(n,k_1,k_2)=(7,4,0)$ took the longest time, which was
approximately 733 hours.

To save space, we only list in Tables~\ref{Tab:1}--\ref{Tab:7} 
the numbers $N'(n,k_1,k_2)$ for $n=1,2,\ldots,7$, respectively.
Generator matrices can be obtained electronically from
``\url{http://yuki.cs.inf.shizuoka.ac.jp/Z4codes/}''.

\begin{table}[thbp]
\caption{Length 1}
\label{Tab:1}
\begin{center}
{\footnotesize
\begin{tabular}{c|c|c|c||c|c|c|c}
\noalign{\hrule height0.8pt}
 $|C|$   & $k_1$  & $k_2$ & $N'(1,k_1,k_2)$&
 $|C|$   & $k_1$  & $k_2$ & $N'(1,k_1,k_2)$\\ \hline
  $2$ & $0$ & $1$ & $1$ &
  $2^2$ & $1$ &  $0$ & $1$\\ 
\noalign{\hrule height0.8pt}
\end{tabular}
}
\end{center}
\end{table}

\begin{table}[thbp]
\caption{Length 2}
\label{Tab:2}
\begin{center}
{\footnotesize
\begin{tabular}{c|c|c|c||c|c|c|c}
\noalign{\hrule height0.8pt}
$|C|$   & $k_1$  & $k_2$ & $N'(2,k_1,k_2)$& 
$|C|$   & $k_1$  & $k_2$ & $N'(2,k_1,k_2)$\\ 
\hline
  $2$ & $0$ & $1$ & $1$&    $2^3$ & 1 & 1 &2 \\ 
  {$2^2$} & $0$ & $2$ & $1$&    $2^4$ & 2 & 0& 1\\ 
          & $1$ & $0$ & $2$&    &&&\\
\noalign{\hrule height0.8pt}
\end{tabular}
}
\end{center}
\end{table}

\begin{table}[thbp]
\caption{Length 3}
\label{Tab:3}
\begin{center}
{\footnotesize
\begin{tabular}{c|c|c|c||c|c|c|c}
\noalign{\hrule height0.8pt}
$|C|$ & $k_1$  & $k_2$ & $N'(3,k_1,k_2)$ &
$|C|$ & $k_1$  & $k_2$ & $N'(3,k_1,k_2)$\\ \hline
$2$ & $0$ & $1$ &1 &$2^4$ & 1 & 2 & 3 \\
$2^2$ & $0$ & $2$ &2 &      & 2 & 0 & 5 \\
      & $1$ & $0$ &3 &$2^5$ & 2 & 1 & 3 \\
$2^3$ & 0 & 3 &1     &$2^6$ & 3 & 0 & 1 \\ 
      & 1 & 1 & 7    &&&&\\
\noalign{\hrule height0.8pt}
\end{tabular}
}
\end{center}
\end{table}

\begin{table}[thbp]
\caption{Length 4}
\label{Tab:4}
\begin{center}
{\footnotesize
\begin{tabular}{c|c|c|c||c|c|c|c}
\noalign{\hrule height0.8pt}
$|C|$   & $k_1$  & $k_2$ & $N'(4,k_1,k_2)$&
$|C|$   & $k_1$  & $k_2$ & $N'(4,k_1,k_2)$\\ \hline
$2$ & $0$ & $1$ & 1& $2^5$ & 1 & 3 & 4 \\ 
$2^2$ & $0$ & $2$ & 3&       & 2 & 1 & 23 \\
      & $1$ & $0$ & 4& $2^6$ & 2 & 2 & 6 \\ 
$2^3$ & 0 & 3 & 3    &       & 3 & 0 & 9 \\ 
      & 1 & 1 &17    & $2^7$ & 3 & 1 & 4 \\ 
$2^4$ & 0 & 4 & 1    & $2^8$ & 4 & 0 & 1 \\ 
      & 1 & 2 & 16   & &&&\\
      & 2 & 0 & 18   & &&&\\
\noalign{\hrule height0.8pt}
\end{tabular}
}
\end{center}
\end{table}

\begin{table}[thbp]
\caption{Length 5}
\label{Tab:5}
\begin{center}
{\footnotesize
\begin{tabular}{c|c|c|c||c|c|c|c}
\noalign{\hrule height0.8pt}
$|C|$   & $k_1$  & $k_2$ & $N'(5,k_1,k_2)$ &
$|C|$   & $k_1$  & $k_2$ & $N'(5,k_1,k_2)$\\ \hline
$2$ & $0$ & $1$ & 1&$2^6$ & 1 & 4 & 5  \\
$2^2$ & $0$ & $2$ & 4&      & 2 & 2 & 67 \\
      & $1$ & $0$ & 5&      & 3 & 0 & 63 \\
$2^3$ & 0 & 3 & 6    &$2^7$ & 2 & 3& 10  \\
      & 1 & 1 &33    &      & 3 & 1 &55  \\
$2^4$ & 0 & 4 & 4    &$2^8$ & 3 & 2 & 10 \\
      & 1 & 2 & 54   &      & 4 & 0 &14  \\
      & 2 & 0 & 49   &$2^9$ & 4 & 1 & 5  \\
$2^5$ & 0 & 5 &1     &$2^{10}$& 5 & 0 & 1\\
      & 1 & 3 & 29   &  &&&\\
      & 2 & 1 & 121  &  &&&\\
\noalign{\hrule height0.8pt}
\end{tabular}
}
\end{center}
\end{table}

\begin{table}[thbp]
\caption{Length 6}
\label{Tab:6}
\begin{center}
{\footnotesize
\begin{tabular}{c|c|c|c||c|c|c|c}
\noalign{\hrule height0.8pt}
$|C|$   & $k_1$  & $k_2$ & $N'(6,k_1,k_2)$ &
$|C|$   & $k_1$  & $k_2$ & $N'(6,k_1,k_2)$\\ \hline
$2$ & $0$ & $1$ & 1& $2^7$ & 1 & 5 & 6 \\
$2^2$ & $0$ & $2$ &6 &       & 2 & 3  & 157 \\
      & $1$ & $0$ & 6&       & 3 & 1 &587 \\ 
$2^3$ & 0 & 3 & 12   & $2^8$ & 2 & 4 & 16 \\ 
      & 1 & 1 &58    &       & 3 & 2 & 212 \\
$2^4$ & 0 & 4 &11    &       & 4 & 0 &179 \\ 
      & 1 & 2 &149   & $2^9$ & 3 & 3 & 22 \\ 
      & 2 & 0 &121   &       & 4 & 1 & 112 \\
$2^5$ & 0 & 5 &5     & $2^{10}$ & 4 & 2 & 16 \\
      & 1 & 3 & 134  &           & $5$ & 0 & 20 \\
      & 2 & 1 & 499  & $2^{11}$ & 5 & 1 & 6 \\ 
$2^6$ & 0 & 6 & 1    & $2^{12}$ & 6 & 0 & 1 \\ 
      & 1 & 4  & 47  & &&&\\
      & 2 & 2 &500   & &&&\\
      & 3 & 0 & 381  & &&&\\
\noalign{\hrule height0.8pt}
\end{tabular}
}
\end{center}
\end{table}

\begin{table}[thbp]
\caption{Length 7}
\label{Tab:7}
\begin{center}
{\footnotesize
\begin{tabular}{c|c|c|c||c|c|c|c}
\noalign{\hrule height0.8pt}
$|C|$   & $k_1$  & $k_2$ & $N'(7,k_1,k_2)$&
$|C|$   & $k_1$  & $k_2$ & $N'(7,k_1,k_2)$\\ \hline
$2$ & $0$ & $1$ & 1&$2^8$ & 1 & 6 & 7 \\ 
$2^2$ & $0$ & $2$ & 7&      & 2 & 4 & 319 \\
      & $1$ & $0$ & 7&      & 3 & 2  & 3247 \\
$2^3$ & 0 & 3 & 21   &      & 4 & 0  & 2215 \\
      & 1 & 1 &93    &$2^9$ & 2 & 5 & 23 \\ 
$2^4$ & 0 & 4 &27    &      & 3 & 3 & 648\\ 
      & 1 & 2 &359   &      & 4 & 1 &2257 \\ 
      & 2 & 0 &256   &$2^{10}$ & 3 & 4 & 43 \\ 
$2^5$ & 0 & 5 &17    &         & 4 & 2 & 565 \\
      & 1 & 3 & 503  &          & 5 & 0 & 429\\ 
      & 2 & 1 & 1728 &$2^{11}$ & 4 & 3& 43 \\ 
$2^6$ & 0 & 6 & 6    &         & 5 & 1 & 204 \\
      & 1 & 4 & 283  &$2^{12}$ & 5 & 2 & 23 \\ 
      & 2 & 2 & 2896 &         & 6 & 0 & 27 \\ 
      & 3 & 0 & 1955 &$2^{13}$ & 6 & 1 & 7 \\
$2^7$ & 0 & 7 & 1    &$2^{14}$ & 7 & 0 & 1\\
      & 1 & 5 & 70   & &&&\\
      & 2 & 3  &1582 & &&&\\
      & 3 & 1 & 5184 & &&&\\
\noalign{\hrule height0.8pt}
\end{tabular}
}
\end{center}
\end{table}

As a summary, we list the numbers $N'(n)$ of inequivalent $\ZZ_4$-codes 
of lengths $n \le 7$, none of which is equivalent
to the trivial extension of a $\ZZ_4$-code of length $n-1$.

\begin{prop}
$N'(1)=2$,
$N'(2)=7$,
$N'(3)=26$,
$N'(4)=110$,
$N'(5)=537$,
$N'(6)=3265$ and 
$N'(7)=25054$.
\end{prop}

The number $N(n,k_1,k_2)$ of inequivalent $\ZZ_4$-codes of length $n$
and type $4^{k_1}2^{k_2}$ can be obtained from
$N'(n,k_1,k_2)$ and $N(n-1,k_1,k_2)$ by Proposition~\ref{prop:trivial}.
As a check, we verified that the numbers $N(n,k_1,k_2)$ 
satisfy~\eqref{eq:number} for $n \le 7$.
As a summary, we list the numbers $N(n)$ of inequivalent $\ZZ_4$-codes 
of lengths $n \le 7$.

\begin{prop}
$N(1)=2$,
$N(2)=9$,
$N(3)=35$,
$N(4)=145$,
$N(5)=682$,
$N(6)=3947$ and
$N(7)=29001$.
\end{prop}

We end this note with a certain remark.
There is another approach to completing the classification 
of $\ZZ_4$-codes of length $n$.
If a classification of  $\ZZ_4$-codes of length $n$ and type $4^{k_1}2^{k_2}$
is done for only $k_1 \le n-k_1-k_2$, then
the remaining classification is obtained from the above classification
by considering their dual codes (see~\eqref{eq:number}).
In this note, we employed the approach given 
in Section~\ref{sec:method}, 
because~\eqref{eq:number} can be used as a check.

\bigskip
\noindent
{\bf Acknowledgments.}
This work was supported by JSPS KAKENHI Grant Numbers 
26610032, 15H03633.
The authors would like to thank Akihiro Munemasa for 
helpful comments.



\end{document}